\documentclass{article}
\usepackage{amsfonts}
\usepackage{amsmath}
\usepackage{fancyhdr}

\newtheorem{theorem}{Theorem}

\newtheorem{corollary}[theorem]{Corollary}

\newtheorem{example}[theorem]{Example}

\newtheorem{lemma}[theorem]{Lemma}

\newtheorem{proposition}[theorem]{Proposition}

\newenvironment{proof}[1][Proof]{\noindent\textbf{#1.} }{\ \rule{0.5em}{0.5em}}
\input{tcilatex}
\begin{document}

\date{\today}
\title{\textbf{Totally Geodesic Submanifolds in Tangent Bundle with
g-natural Metric}}
\author{Stanis\l aw Ewert-Krzemieniewski (Szczecin)}
\maketitle

\begin{abstract}
In the paper we investigate submanifolds in a tangent bundle endowed with
g-natural metric $G$, defined by a vector field on a base manifold. We give
a sufficient condition for a vector field on $M$ to defined totally geodesic
submanifold in ($TM,G).$ The parallel vector field is discussed in more
detail.

\textbf{Mathematics Subject Classification }Primary 53B25, 53C40, secondary
53B21, 55C25.

\textbf{Key words}: Riemannian manifold, submanifold, tangent bundle, g -
natural metric, totally geodesic, Sasaki metric.
\end{abstract}

\section{Preliminaries}

A smooth vector field $u$ on a manifold $M$ defines a submanifold $u(M)$ in
the tangent bundle $TM$ in an obvious way. A set of early results on the
subject is presented in (\cite{1}, Chapter III). The metric on $u(M),$ if
any is considered, is induced from the complete or vertical lifts of the
metric on $M.$ On the other hand, there are few results on totally geodesic
submanifolds in $TM$ defined by a vector field on $M,$ with metric induced
from the Sasaki one. First, if $u$ is a zero vector field on $M,$ then $u(M)$
is totally geodesic (\cite{2}). For $u\neq 0$ Walczak proved

\begin{theorem}
(\cite{3})

\begin{enumerate}
\item If $\nabla u=0$ on $(M,g),$ then $u(M)$ is a totally geodesics
submanifold.

\item If $u$ is of constant length and $u(M)$ is totally geodesic, then $u$
is parallel $(\nabla u=0).$
\end{enumerate}
\end{theorem}

For submanifolds in ($TM,G_{S})$ we have

\begin{theorem}
(\cite{4}, Proposition 2.1, Theorem 2.1)

\begin{enumerate}
\item Let $N^{m}\subset TM,$ be a submanifold embedded in the tangent bundle
of a Riemannian manifold $M,$ which is transverse to the fibre at a point $%
z\in N^{m}.$

Then there is a submanifold $F^{m}\subset M,$ such that $\pi (z)\in F^{m},$
an open $U\ $containing $\pi (z),$ an open $V$ containing $z\in TM$ and a
vector field $u$ along $F^{m}\cap U$ such that $N^{m}\cap V=u(F^{m}\cap U).$

\item Let $N^{m}\ $be a connected compact submanifold of the tangent bundle
of a connected simply connected Riemannian manifold $M^{m}$ that is
everywhere transverse to the fibres of $TM^{m}.$

Then $M^{m}$ is compact and there is a vector field $u$ on $M^{m},$ such
that 
\begin{equation*}
u(M^{m})=N^{m}.
\end{equation*}
\end{enumerate}
\end{theorem}

In this respect the former theorem was extended to the case of a vector
field $u$ defined on an arbitrary submanifold of $M$ (\cite{4}, Proposition
3.1 and Corollaries 3.1 - 3.3) to give necessary and sufficient conditions
for $u(M)$ to be totally geodesic.

The aim of the paper is to extend the above results to the case of a tangent
bundle endowed with $g-$ natural metric $G.$

Firstly, we shall indicate large classes of the $g-$ natural metrics such
that the theorems by Walczak cited above hold. Secondly, we give a
sufficient condition for a vector field $u$ to define totally geodesic
submanifold in ($TM,G).$ At the end of the paper some examples are given.

Throughout the paper all manifolds under consideration are smooth and
Hausdorff ones. The metric $g$ of the base manifold $M$ is always assumed to
be Riemannian one.

\section{Tangent Bundle}

Let $x$ be a point of a Riemannian manifold $(M,g),$ dim$M=n,$ covered by
coordinate neighbourhoods $(U,$ $(x^{j})),$ $j=1,...,n.\ $Let $TM\ $be a
tangent bundle of $M$ and $\pi :TM\longrightarrow M\ $be a natural
projection on $M.$ If $x\in U$ and $u=u^{r}\frac{\partial }{\partial x^{r}}%
_{\mid x}\in T_{x}M\ $then $(\pi ^{-1}(U),$ $((x^{r}),(u^{r})),$ $r=1,...,n,$
is a coordinate neighbourhood on $TM.$ We shall write $\partial _{k}=\frac{%
\partial }{\partial x^{k}}$ and $\delta _{k}=\frac{\partial }{\partial u^{k}}%
.$

The space $T_{(x,u)}TM$ tangent to $TM$ at $(x,u)$ splits into a direct sum
of two isomorphic subspaces%
\begin{equation*}
T_{(x,u)}TM=H_{(x,u)}TM\oplus V_{(x,u)}TM.
\end{equation*}

$V_{(x,u)}TM$ is the kernel of the differential of the projection $\pi
:TM\longrightarrow M,$ i.e. 
\begin{equation*}
V_{(x,u)}TM=Ker\left( d\pi |_{(x,u)}\right)
\end{equation*}%
and is called the vertical subspace of $T_{(x,u)}TM.$

$H_{(x,u)}TM$ is the kernel of the connection map of the Levi-Civita
connection $\nabla ,$ i.e.%
\begin{equation*}
H_{(x,u)}TM=Ker(K_{(x,u)})
\end{equation*}%
and is called the horizontal subspace.

For any smooth vector field $Z:M\longrightarrow TM$ and $X_{x}\in T_{x}M$
the connection map%
\begin{equation*}
K_{(x,u)}:T_{(x,u)}TM\longrightarrow T_{x}M
\end{equation*}%
of the Levi-Civita connection $\nabla $ \ is given by 
\begin{equation*}
K(dZ_{x}(X_{x}))=\left( \nabla _{X}Z\right) _{x}.
\end{equation*}%
If $\widetilde{X}=\left( X^{r}\partial _{r}+\overline{X}^{r}\delta
_{r}\right) |_{(x,u)},$ then $d\pi \left( \widetilde{X}\right)
=X^{r}\partial _{r}|_{x}$ and $K(\widetilde{X})=(\overline{X}%
^{r}+u^{s}X^{t}\Gamma _{st}^{r})\partial _{r}|_{x}.$ On the other hand, for
any vector field $X=X^{r}\partial _{r}$ on $M$ there exist unique vectors 
\begin{equation*}
X^{h}=X^{j}\partial _{j}-u^{r}X^{s}\Gamma _{rs}^{j}\delta _{j},\quad
X^{v}=X^{j}\delta _{j},
\end{equation*}%
called the horizontal and vertical lifts of $X$ respectively. We have%
\begin{equation*}
d\pi (X^{h})=X,\quad K(X^{h})=0,\quad d\pi (X^{v})=0,\quad K(X^{v})=X.
\end{equation*}%
In (\cite{5}) the class of $g-$natural metrics was defined. We have

\begin{lemma}
(\cite{5},\cite{6}, \cite{7}) Let $(M,g)$ be a Riemannian manifold and $G$
be a $g-$natural metric on $TM.$ There exist functions $a_{j},$ $%
b_{j}:<0,\infty )\longrightarrow R,$ $j=1,2,3,$ such that for every $X,$ $Y,$
$u\in T_{x}M$%
\begin{multline*}
G_{(x,u)}(X^{h},Y^{h})=(a_{1}+a_{3})(r^{2})g_{x}(X,Y)+(b_{1}+b_{3})(r^{2})g_{x}(X,u)g_{x}(Y,u),
\\
G_{(x,u)}(X^{h},Y^{v})=G_{(x,u)}(X^{v},Y^{h})=a_{2}(r^{2})g_{x}(X,Y)+b_{2}(r^{2})g_{x}(X,u)g_{x}(Y,u),
\\
G_{(x,u)}(X^{v},Y^{v})=a_{1}(r^{2})g_{x}(X,Y)+b_{1}(r^{2})g_{x}(X,u)g_{x}(Y,u),
\end{multline*}%
where $r^{2}=g_{x}(u,u).$ For $\dim M=1$ the same holds for $b_{j}=0,$ $%
j=1,2,3.$
\end{lemma}

Setting $a_{1}=1,$ $a_{2}=a_{3}=b_{j}=0$ we obtain the Sasaki metric, while
setting $a_{1}=b_{1}=\frac{1}{1+r^{2}},$ $a_{2}=b_{2}=0=0,$ $a_{1}+a_{3}=1,$ 
$b_{1}+b_{3}=1$ we get the Cheeger-Gromoll one.

Following (\cite{6}) we put

\begin{enumerate}
\item $a(t)=a_{1}(t)\left( a_{1}(t)+a_{3}(t)\right) -a_{2}^{2}(t),$

\item $F_{j}(t)=a_{j}(t)+tb_{j}(t),$

\item $F(t)=F_{1}(t)\left[ F_{1}(t)+F_{3}(t)\right] -F_{2}^{2}(t)$

for all $t\in <0,\infty ).$
\end{enumerate}

We shall often abbreviate: $A=a_{1}+a_{3},$ $B=b_{1}+b_{3}.$

\begin{lemma}
\label{Lemma 9}(\cite{6}, Proposition 2.7) The necessary and sufficient
conditions for a $g-$ natural metric $G$ on the tangent bundle of a
Riemannian manifold $(M,g)$ to be non-degenerate are $a(t)\neq 0$ and $%
F(t)\neq 0$ for all $t\in <0,\infty ).$ If $\dim M=1$ this is equivalent to $%
a(t)\neq 0$ for all $t\in <0,\infty ).$
\end{lemma}

\section{Submanifolds Defined by Vector Fields}

Let $u$ be a smooth vector field on a Riemannian manifold $(M,g)$ and $%
N=u(M).$ If $p\in M,$ then $z=(p,u(p))\in TM.$ The space $%
T_{z}N=T_{(p,u(p))}N$ consists of vectors $u_{\ast }(X)=X^{h}+(\nabla
_{X}u)^{v}$ for all $X\in T_{p}M.$ The normal space $T_{z}^{\perp }N$
consists of vectors $\eta =W^{h}+V^{v},$ where $W,$ $V\in T_{p}M$ and 
\begin{equation}
G_{z}(\eta ,u_{\ast }(X))=0.
\end{equation}%
Hence, we have%
\begin{multline}
g(AW+Bg(u,W)u+a_{2}V+b_{2}g(u,V)u,\ X)+  \label{20} \\
g(a_{1}V+b_{1}g(u,V)u+a_{2}W+b_{2}g(u,W)u,\ \nabla _{X}u)=0
\end{multline}%
for all $X\in \mathfrak{X}(M).$

By the use of the Weingarten formula%
\begin{equation*}
\widetilde{\nabla }_{X}V=-A_{V}X+D_{X}V,
\end{equation*}%
the definition of the $g$ - natural metric $G$ and its Levi-Civita
connection $\widetilde{\nabla }$ (\cite{6}, \cite{7}, for definitions \ of F
tensors cf \cite{8}) we obtain%
\begin{multline}
-G_{z}(A_{W^{h}+V^{v}}u_{\ast }(X),\ u_{\ast }(X))=G(\widetilde{\nabla }%
_{u_{\ast }(X)}(W^{h}+V^{v}),\ u_{\ast }(X))=  \label{30} \\
Ag(\nabla _{X}W,X)+Bg(\nabla _{X}W,u)g(X,u)+a_{2}g(\nabla _{X}W,\nabla
_{X}u)+ \\
b_{2}g(\nabla _{X}W,u)g(\nabla _{X}u,u)+a_{2}g(\nabla _{X}V,X)+b_{2}g(\nabla
_{X}V,u)g(X,u)+ \\
a_{1}g(\nabla _{X}V,\nabla _{X}u)+b_{1}g(\nabla _{X}V,u)g(\nabla _{X}u,u)+ \\
A\left\{ A[u,X,W,X]+C[u,X,V,X]+C[u,W,\nabla _{X}u,X]+\right. \left.
E[u,\nabla _{X}u,V,X]\right\} + \\
B\left\{ A[u,X,W,u]+C[u,X,V,u]+C[u,W,\nabla _{X}u,u]+\right. \left.
E[u,\nabla _{X}u,V,u]\right\} g(X,u)+ \\
a_{2}\left\{ A[u,X,W,\nabla _{X}u]+C[u,X,V,\nabla _{X}u]+C[u,W,\nabla
_{X}u,\nabla _{X}u]+\right. \\
\left. E[u,\nabla _{X}u,V,\nabla _{X}u]\right\} + \\
b_{2}\left\{ A[u,X,W,u]+C[u,X,V,u]+C[u,W,\nabla _{X}u,u]+\right. \\
\left. E[u,\nabla _{X}u,V,u]\right\} g(u,\nabla _{X}u)+ \\
a_{2}\left\{ B[u,X,W,X]+D[u,X,V,X]+D[u,W,\nabla _{X}u,X]+F[u,\nabla
_{X}u,V,X]\right\} + \\
b_{2}\left\{ B[u,X,W,u]+D[u,X,V,u]+D[u,W,\nabla _{X}u,u]+F[u,\nabla
_{X}u,V,u]\right\} g(X,u)+ \\
a_{1}\left\{ B[u,X,W,\nabla _{X}u]+D[u,X,V,\nabla _{X}u]+D[u,W,\nabla
_{X}u,\nabla _{X}u]+\right. \\
\left. F[u,\nabla _{X}u,V,\nabla _{X}u]\right\} + \\
b_{1}\left\{ B[u,X,W,u]+D[u,X,V,u]+D[u,W,\nabla _{X}u,u]+\right. \\
\left. F[u,\nabla _{X}u,V,u]\right\} g(u,\nabla _{X}u).
\end{multline}%
Hence, after straightforward long computations, we obtain

\begin{lemma}
\label{L10}Let $u$ be a vector field on a Riemannian manifold $(M,g).$ Let $%
u(M$) be a submanifold defined by $u$ in $TM$ with metric induced from a $g$
- natural non-degenerated metric $G$ on $TM.$ Then the second fundamental
form $A$ satisfies%
\begin{multline}
-G_{z}(A_{W^{h}+V^{v}}u_{\ast }(X),\ u_{\ast }(X))=G(\widetilde{\nabla }%
_{u_{\ast }(X)}(W^{h}+V^{v}),\ u_{\ast }(X))=  \label{40} \\
a_{1}R(u,\nabla _{X}u,W,X)+a_{2}R(u,X,W,X)+ \\
Ag(X,\nabla _{X}W)+Bg(u,X)g(u,\nabla _{X}W)+Bg(u,X)g(V,X)+ \\
a_{1}g(\nabla _{X}u,\nabla _{X}V)+b_{1}g(u,\nabla _{X}u)g(V,\nabla
_{X}u)+b_{1}g(u,\nabla _{X}u)g(u,\nabla _{X}V)+ \\
a_{2}g(X,\nabla _{X}V)+a_{2}g(\nabla _{X}u,\nabla _{X}W)+ \\
b_{2}g(u,X)g(u,\nabla _{X}V)+b_{2}g(u,X)g(V,\nabla _{X}u)+b_{2}g(u,\nabla
_{X}u)g(V,X)+ \\
b_{2}g(u,\nabla _{X}u)g(u,\nabla _{X}W)+ \\
A^{\prime }g(u,V)g(X,X)+B^{\prime }g(u,V)g(u,X)^{2}+a_{1}^{\prime
}g(u,V)g(\nabla _{X}u,\nabla _{X}u)+ \\
2a_{2}{}^{\prime }g(u,V)g(X,\nabla _{X}u)+b_{1}^{\prime }g(u,V)g(u,\nabla
_{X}u)^{2}+2b_{2}^{\prime }g(u,V)g(u,X)g(u,\nabla _{X}u)
\end{multline}%
at a point $z=(p,u(p))\in TM,$ $p\in M.$
\end{lemma}

If $u\ $is a parallel vector field on $M,\ $i.e. $\nabla _{X}u=0$ for all $%
X\in \mathfrak{X}(M),$ then from (\ref{20}) we get

\begin{equation*}
AW+Bg(u,W)u+a_{2}V+b_{2}g(u,V)u=0,
\end{equation*}%
whence

\begin{equation*}
A\nabla _{X}W+Bg(u,\nabla _{X}W)u+a_{2}\nabla _{X}V+b_{2}g(u,\nabla
_{X}V)u=0,
\end{equation*}
since $\nabla _{X}f(r^{2})=2f^{\prime }(r^{2})g(u,\nabla _{X}u)=0.$ Applying
to Lemma \ref{L10} we get%
\begin{multline*}
-G_{z}(A_{W^{h}+V^{v}}u_{\ast }(X),\ u_{\ast }(X))=a_{2}g\left( \QATOP{{}}{{}%
}R(u,X,W),\ X\right) + \\
g\left( \QATOP{{}}{{}}Bg(u,X)X+A^{\prime }g(X,X)u+B^{\prime }g(u,X)^{2}u,\
V\right) .
\end{multline*}%
Consequently, we have

\begin{proposition}
Let $u$ be a parallel vector field on a Riemannian manifold $(M,g).$ If $G$
is a non-degenerated $g-$ natural metric on $TM,$ such that $A^{\prime }=0,$ 
$B=0$ and either $a_{2}=0$ \ or $M$ is flat, then $u(M)$ is a totally
geodesic submanifold in $TM.$
\end{proposition}

\begin{proposition}
Let $G$ be a non-degenerated $g-$ natural metric on $TM$ such that $%
A^{\prime }=0,$ $B=0,$ $a_{2}=b_{2}=0$ on $TM$ and $a_{1}+a_{1}^{\prime
}r^{2}\neq 0$ on a dense subset of $TM.$

If $u$ is of constant length and $u(M)$ is a totally geodesic submanifold in 
$(M,g)$, then $u$ is parallel.
\end{proposition}

\begin{proof}
Since%
\begin{equation*}
G(u_{\ast }(X),\ (F_{1}+F_{3})u^{v}-F_{2}u^{h})=Fg(u,\nabla _{X}u),
\end{equation*}%
the vector field $\tau =(F_{1}+F_{3})u^{v}-F_{2}u^{h}$ is normal if and only
if $u$ is of constant length, i.e. $g(u,\nabla _{X}u)=0.$ Thus Lemma \ref%
{L10} yields%
\begin{multline*}
-G_{z}(A_{\tau }u_{\ast }(X),\ u_{\ast }(X))= \\
\left[ \left( a_{1}+a_{1}^{\prime }r^{2}\right) (F_{1}+F_{3})-a_{2}F_{2}%
\right] g(\nabla _{X}u,\nabla _{X}u)- \\
F_{2}\left[ a_{1}R(u,\nabla _{X}u,u,X)+a_{2}R(u,X,u,X)+Ag(X,\nabla _{X}u)%
\right] + \\
(F_{1}+F_{3})\left[ (B+B^{\prime }r^{2})g(u,X)^{2}+(a_{2}+a_{2}^{\prime
}r^{2})g(X,\nabla _{X}u)+A^{\prime }r^{2}g(X,X)\right] .
\end{multline*}%
This completes the proof.
\end{proof}

\begin{corollary}
In particular, this holds for $G$ being either Sasaki or Cheeger-Gromoll
metric.
\end{corollary}

\section{Generalization}

Put 
\begin{multline*}
T_{W}(X,u)=\nabla _{X}\left( \QDATOP{{}}{{}}AX+a_{2}\nabla
_{X}u+Bg(u,X)u+b_{2}g(u,\nabla _{X}u)u\right) + \\
a_{1}R(u,\nabla _{X}u,X)+a_{2}R(u,X,X),
\end{multline*}%
\begin{multline*}
T_{V}(X,u)=\nabla _{X}(\QDATOP{{}}{{}}a_{2}X+a_{1}\nabla _{X}u)+b_{2}\nabla
_{X}(g(u,X))u+\nabla _{X}(b_{1}g(u,\nabla _{X}u))u- \\
\left( \QDATOP{{}}{{}}Bg(u,X)+b_{2}g(u,\nabla _{X}u)\right) X- \\
\left( \QDATOP{{}}{{}}\frac{1}{2}\nabla _{X}(b_{1})g(u,\nabla
_{X}u)+A^{\prime }g(X,X)+B^{\prime }g(u,X)^{2}+\right. \\
\left. \QDATOP{{}}{{}}a_{1}^{\prime }g(\nabla _{X}u,\nabla
_{X}u)+2a_{2}^{\prime }g(X,\nabla _{X}u)\right) u.
\end{multline*}

\begin{proposition}
For any vector fields $X,W,V$ tangent to $(M,g)$ such that $W^{h}+V^{v}$ is
normal we have%
\begin{equation}
G_{z}(A_{W^{h}+V^{v}}u_{\ast }(X),\ u_{\ast
}(X))=g(W,T_{W}(X,u))+g(V,T_{V}(X,u)).  \label{60}
\end{equation}
\end{proposition}

\begin{proof}
To prove the proposition one has to differentiate (\ref{20}) with respect to
the vector field $X,$ then subtract (\ref{40}) to eliminate terms containing 
$\nabla _{X}W,$ $\nabla _{X}V$ and, finally, apply the identity $g(\nabla
_{X}u,\nabla _{X}u)=X(g(u,\nabla _{X}u))-g(u,\nabla _{X}\nabla _{X}u).$
\end{proof}

\begin{corollary}
The submanifold $u(M)$ defined in $(TM,G)$ by a smooth vector field $u$ on a
Riemannian manifold $M$ is totally geodesic if $T_{W}(X,u)=T_{V}(X,u)=0\ $%
for all vector fields $X$ tangent to $(M,g).$
\end{corollary}

Observe that if $u$ is parallel, then replacing in (\ref{20}) $X$ with $%
\nabla _{X}X$ we obtain%
\begin{equation*}
g(AW+Bg(u,W)u+a_{2}V+b_{2}g(u,V)u,\nabla _{X}X)=0.
\end{equation*}

Moreover, if $u$ is a tors-forming vector field on $M$, i.e. $\nabla
_{X}u=\varrho (X)u+\alpha X$ for all $X\in TM,$ then%
\begin{multline*}
g\left[ W,(A+\alpha a_{2})X+(B+\alpha b_{2})g(u,X)u+\rho (X)F_{2}u\right] +
\\
g\left[ V,(a_{2}+\alpha a_{1})X+(b_{2}+\alpha b_{1})g(u,X)u+\rho (X)F_{1}u%
\right] =0.
\end{multline*}

Replacing in the last equation (or in (\ref{20})) an arbitrary $X$ with $%
\nabla _{X}X$ and subtracting the result from (\ref{60}) we shall get for
the particular vector field $u$ the following

\begin{proposition}
If $u$ is a concircular vector field on $M$, i.e. $\nabla _{X}u=\alpha X$ \
for all $X\in TM,$ then%
\begin{multline*}
T_{W}(X,u)=\left[ X(A+\alpha a_{2})+\alpha (B+\alpha b_{2})g(u,X)\right]
X+(a_{2}+\alpha a_{1})R(u,X,X)+ \\
\left[ X(B+\alpha b_{2})g(u,X)+\alpha (B+\alpha b_{2})g(X,X)\right] u,
\end{multline*}%
\begin{multline*}
T_{V}(X,u)=\left[ X(a_{2}+\alpha a_{1})-(B+\alpha b_{2})g(u,X)\right] X+ \\
\left[ \left( b_{1}X(\alpha )+\frac{1}{2}\alpha X(b_{1})\right)
g(u,X)+\alpha (b_{2}+\alpha b_{1})g(X,X)\right] u- \\
\left[ B^{\prime }g(u,X)^{2}+(A^{\prime }+\alpha (a_{1}^{\prime }\alpha
+2a_{2}^{\prime }))g(X,X)\right] u.
\end{multline*}
\end{proposition}

For the Sasaki metric and a concircular vector field $u$ we have 
\begin{equation*}
T_{W}(X,u)=\alpha R(u,X,X),\quad T_{V}(X,u)=0,
\end{equation*}%
so the concircular non-parallel vector field gives rise to a totally
geodesic submanifold only if $(M,g)$ is flat.

On the other hand, for a given concircular vector field on a non-flat $%
(M,g), $ we can always construct a family of nondegenerated $g$ - natural
metrics on $TM$ that are not Sasaki ones so that $u(M)$ is totally geodesic
in $TM$. Having given $\alpha $ and, for example, $a_{1}$ on a neighbourhood
of a point $z=(p,u(p))\in TM$ it is enough to put $a_{2}=-\alpha a_{1},$ $%
A=\alpha ^{2}a_{1}+C,$ $C=$ $const\neq 0,$ $b_{j}=0.$ Similarly, if $(M,g)$
is flat and $\alpha =const.$one puts $a_{2}=-\alpha a_{1}+C_{1},$ $A=\alpha
^{2}a_{1}+C_{1}\alpha +C,$ $C=$ $const,$ $C_{1}=const,$ $Ca_{1}+C_{1}\neq 0,$
$b_{j}=0.$

\begin{proposition}
If $u$ is a recurrent vector field on $M$, i.e. $\nabla _{X}u=\varrho (X)u$
for all $X\in TM,$ then%
\begin{multline*}
T_{W}(X,u)=F_{2}\left[ (\nabla _{X}\rho )(X)+\rho ^{2}\right] u+2A^{\prime
}\rho r^{2}X+2\rho (B+B^{\prime }r^{2})g(u,X)u+ \\
2\rho ^{2}r^{2}(a_{2}^{\prime }+b_{2}+b_{2}^{\prime }r^{2})u+a_{2}R(u,X,X),
\end{multline*}%
\begin{multline*}
T_{V}(X,u)=F_{1}\left[ (\nabla _{X}\rho )(X)+\rho ^{2}\right] u+\left[
b_{1}+a_{1}^{\prime }+b_{1}^{\prime }r^{2}\right] \rho ^{2}r^{2}u+ \\
\left[ (2a_{2}^{\prime }-b_{2})\rho r^{2}-Bg(u,X)\right] X-\left[ \rho
(2a_{2}^{\prime }-b_{2})g(u,X)-A^{\prime }g(X,X)u-B^{\prime }g(u,X)^{2}%
\right] u.
\end{multline*}
\end{proposition}

For the Sasaki metric and a recurrent vector field $u$ we have 
\begin{equation*}
T_{W}(X,u)=0,\quad T_{V}(X,u)=\left[ \left( \nabla _{X}\varrho \right)
(X)+\varrho ^{2}\right] u
\end{equation*}%
on arbitrary $(M,g).$

If $(M,g)$ is flat and the recurrent form $\varrho $ is parallel, then there
exist a family of $g$ - natural metrics on $TM,$ such that $u(M)$ is totally
geodesic in $TM.$

\begin{example}
\begin{eqnarray*}
A &=&const\neq 0,\quad B=0,\quad a_{2}=b_{2}=0, \\
a_{1}+ta_{1}^{\prime }+t(2b_{1}+tb_{1}^{\prime }) &=&0,\quad
F_{1}=a_{1}+tb_{1}\neq 0.
\end{eqnarray*}
\end{example}

Stanis\l aw Ewert-Krzemieniewski

West Pomeranian University of Technology Szczecin

School of Mathematics

Al. Piast\'{o}w 17, 70-310 Szczecin, Poland

e-mail: ewert@zut.edu.pl


\begin{thebibliography}{1}
\bibitem[1]{1} Yano, K., Ishihara, S., Tangent and cotangent bundles, Marcel
Dekker, Inc. New York, 1973.

\bibitem[2]{2} Liu, M. S., Affine maps of tangent bundles with Sasaki
metric, Tensor, N.S., 28 (1974), 34-42.

\bibitem[3]{3} Walczak, P., G., On Totally Geodesic Submanifolds of Tangent
Bundle with Sasaki Metric, Bull. Acad. Pol. Sci, ser. Sci. Math. 28, no.3-4
(1980), 161-165.

\bibitem[4]{4} Abbassi, M. T. K., Yampolsky, A., Transverse totally geodesic
submanifolds in of tangent bundle, Publ. Math. Debrecen 64/1-2 (2004),
129-154.

\bibitem[5]{5} Kowalski, O., Sekizawa, M., Natural transformations of
Riemannian metrics on manifolds to metrics on tangent bundles, A
classification. Bull. Tokyo Gakugei Univ. (4) 40 (1988), 1--29.

\bibitem[6]{6} Abbassi, M. T. K., Sarih, Ma\^{a}ti, On natural metrics on
tangent bundles of Riemannian manifolds, Arch. Math. (Brno) 41 (2005), no.
1, 71--92.

\bibitem[7]{7} Abbassi, M. T. K., Sarih, Ma\^{a}ti, On some hereditary
properties of Riemannian $g$-natural metrics on tangent bundles of
Riemannian manifolds, Differential Geom. Appl. 22 (2005), no. 1, 19--47.

\bibitem[8]{8} Ewert-Krzemieniewski, S., On a classification of Killing
vector fields on a tangent bundle with $g$-natural metric, arXiv:1305:3817v1.
\end{thebibliography}
\end{document}